\documentclass[12pt]{amsart}

\usepackage{amsmath}
\usepackage{amsfonts}
\usepackage{amssymb,enumerate}
\usepackage{amsthm}
\usepackage{fancyhdr}
\usepackage{tikz}
\usetikzlibrary{arrows,matrix}
\usepackage{hyperref}

\newtheorem{lem}{Lemma}
\newtheorem{cor}[lem]{Corollary}
\newtheorem{prop}[lem]{Proposition}
\newtheorem{thm}[lem]{Theorem}
\newtheorem{que}[lem]{Question}
\theoremstyle{definition}

\newtheorem{defins}[lem]{Definitions}
\newtheorem{rem}[lem]{Remark}
\newtheorem{ex}[lem]{Example}

\newcommand{\ap}{\mathrm {Ap}}
\newcommand{\ord}{\mathrm {ord}}

\numberwithin{lem}{section}

\title{Classes of complete intersection numerical semigroups}

\author{Marco D'Anna, Vincenzo Micale \and Alessio Sammartano}



\address[M. D'Anna]{Universit\`a di Catania, Dipartimento di Matematica e Informatica, Viale A. Doria, 6, 95125 Catania, Italy}

\email{mdanna@dmi.unict.it}

\address[V. Micale]{Universit\`a di Catania, Dipartimento di Matematica e Informatica, Viale A. Doria, 6, 95125 Catania, Italy}

\email{vmicale@dmi.unict.it}

\address[A. Sammartano]{Department of Mathematics, Purdue University
150 N. University Street, West Lafayette, IN 47907, USA}

\email{asammart@math.purdue.edu}

\begin{document}

\begin{abstract}
We consider several classes of complete intersection numerical
semigroups, arising from many different contexts
like algebraic geometry, commutative algebra, coding theory
and factorization theory.
In particular, we determine
all the logical implications among these classes and provide  examples.
Most of these classes are shown to be well-behaved with respect to the operation of gluing.
\end{abstract}

\maketitle

\section*{Introduction}

The concept of complete intersection is one of the most prominent
in algebraic geometry. The notion of complete intersection for
numerical semigroups (i.e. submonoids of $(\mathbb{N},+)$)
was introduced by Herzog  in \cite{He}, where
he  proved the celebrated theorem stating that a three-generated
semigroup is a complete intersection if and only if it is
symmetric. Complete intersection  semigroups have been studied
extensively since then  (see  e.g. \cite{AG},
\cite{BGRV}, \cite{BGS}, \cite{De},  \cite{RG2}, \cite{Wa}).

Several subclasses of the complete intersections have been
investigated, with different motivations  arising from   algebra
and  geometry. The study of the value-semigroup of  plane
algebroid branches was initiated by Ap\'ery in his famous paper
\cite{A} and then continued by several other authors (e.g.
\cite{BDF}, \cite{Bre}, \cite{Za}). Bertin and Carbonne  defined
free numerical semigroups in \cite{BC} in order to generalize a
formula for the conductor of the local ring of a plane branch in
terms of its Puiseux expansion. Telescopic semigroups were
introduced in \cite{KP} for their applications to codes, but they
have also been studied in connection with homology (cf. \cite{MO})
and factorization theory (cf. \cite{RGG}). Numerical semigroups
with $\beta$-rectangular and $\gamma$-rectangular Ap\'ery set were
defined in \cite{DMS2} to characterize semigroup rings whose
tangent cone is a complete intersection. Finally, semigroups having
a unique Betti element were characterized in \cite{GOR}.

The main purpose of this paper is to understand better the classes
mentioned above and the relations among them. We also introduce a
new class which is naturally related to the previous ones,
semigroups with $\alpha$-rectangular Ap\'ery set. Our main result
is Theorem \ref{main} in which we show that the implications in
Figure \ref{diagram} hold and provide counterexamples for the
``missing arrows''. Some of these implications are somewhat
surprising: despite the fact that the definitions of free and
telescopic semigroups are very similar, two classes of semigroups
with rectangular Ap\'ery sets sit between them. In Section 2 we
study the operation of gluing, which allows to produce new
complete intersection  semigroups from old ones. We show
that semigroups with $\alpha$-rectangular Ap\'ery sets  are also,
in some sense, well-behaved with respect to gluing. We conclude
with some applications to known results in literature.

Computations   were performed by using  GAP
 (cf. \cite{NS},\cite{GAP}).
The tests for the properties treated in this paper will be included in the next release of the
package \texttt{NumericalSgps}.

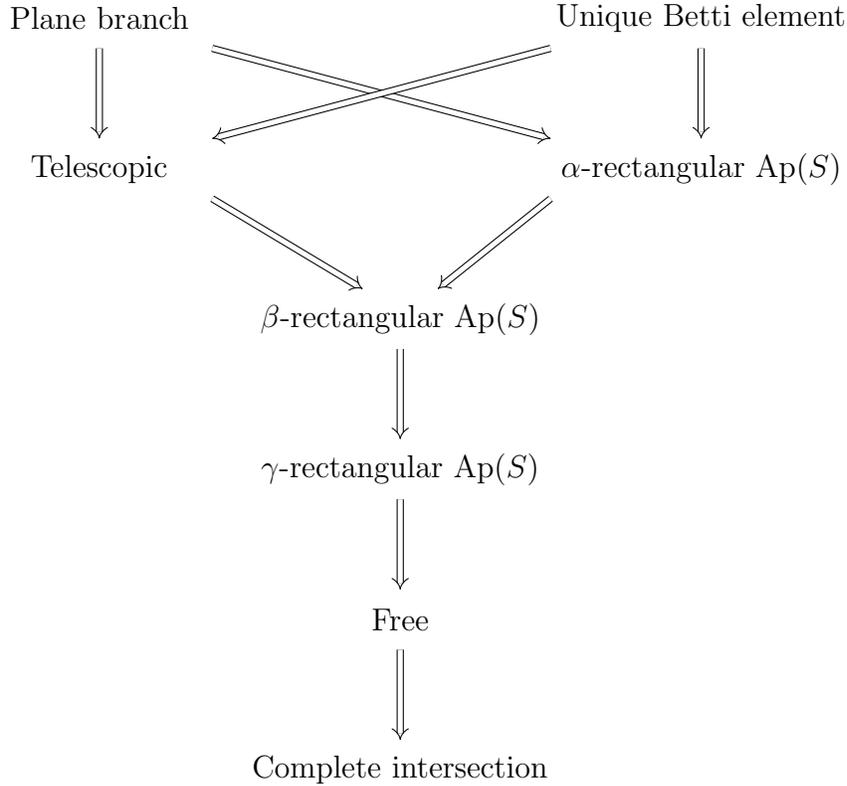
\begin{figure}
\begin{tikzpicture}

\draw  (0,7.6) edge[double,double equal sign distance,-implies] (0,6.4);
\draw [double,double equal sign distance,-implies] (8,7.6) -- (8,6.4);
\draw [double,double equal sign distance,-implies] (1.5,7.6) -- (6,6.4);
\draw [double,double equal sign distance,-implies] (6,7.6) -- (1.5,6.4);
\draw [double,double equal sign distance,-implies] (6,5.6) -- (4.5,4.4);
\draw [double,double equal sign distance,-implies] (1.5,5.6) -- (3.5,4.4);
\draw [double,double equal sign distance,-implies]  (4,3.6) -- (4,2.4);
\draw [double,double equal sign distance,-implies] (4,1.6) -- (4,0.4);
\draw [double,double equal sign distance,-implies] (4,-0.4) -- (4,-1.6);

\node  at (0,8) {Plane branch};
\node  at (8,8) {Unique Betti element};
\node  at (8,6) {$\alpha$-rectangular $\ap(S)$};
\node  at (0,6) {Telescopic};
\node  at (4,4) {$\beta$-rectangular $\ap(S)$};
\node  at (4,2) {$\gamma$-rectangular $\ap(S)$};
\node  at (4,0) {Free};
\node  at (4,-2) {Complete intersection};
\end{tikzpicture}
\caption{Logical implications in Theorem \ref{main}.}
\label{diagram}
\end{figure}

\section{The Classes}

We start by giving some preliminaries on numerical semigroups. Let
$\mathbb{N}$ denote the set of non-negative integers. A {\bf
numerical semigroup} is a subset $ S \subseteq \mathbb{N}$ that is
closed under addition, contains $0$ and has finite complement in
$\mathbb{N}$. The largest integer in $\mathbb{Z} \setminus S$ is
called {\bf Frobenius number} of $S$ and is denoted by $f=f(S)$,
whereas the smallest positive integer in $S$ is known as {\bf
multiplicity} of $S$ and is denoted by $m=m(S)$.

We define a partial order on $S$ setting $ s \preceq t $ if there
is an element $u \in S$ such that $t=u+s$. The set of minimal
elements in the poset $(S \setminus\{0\}, \preceq )$ is called
{\bf minimal system of generators} of $S$.
We define the {\bf embedding dimension} of $S$ as
the cardinality of its minimal system of generators and  denote
it by $\nu=\nu(S)$; it is easy to see that $\nu(S)\leq m(S)$. A
numerical semigroup minimally generated by $\{g_1, \ldots,
g_\nu\}$ will be denoted by $\langle g_1,\ldots,g_\nu \rangle$.
The condition $|\mathbb{N} \setminus S |< \infty $ is equivalent
to $ \gcd (g_1, \ldots, g_\nu)=1$.

For any $n \in S$ we define the {\bf Ap\'ery set of $S$ with
respect to $n$} as $\ap(S,n) =\{s\in S\, |\, s-n \notin S\}$, or
equivalently $\ap(S,n)=\{ \omega_0 , \ldots, \omega_{n-1}\}$ where
$\omega_i = \min \{s \in S : s \equiv i \pmod{n}\}$. The smallest
element in $\ap(S,n)$ is $0$, while the largest one is $f(S) +n$.
If $n=m(S)$ is the multiplicity we just write $\ap(S)$ in place of
$\ap(S,n)$, and we will refer to it simply as the Ap\'ery set of
$S$.

Two types of semigroups are among the most studied, mainly for
their relevance in algebraic geometry. A semigroup $S$ is called
{\bf symmetric} if, for any $x \in \mathbb{Z}$, we have $x \in S
\Leftrightarrow f(S) - x \notin S$; this condition is equivalent
to the fact that $f(S)+m(S)$ is the unique maximal element of the
poset $(\ap(S), \preceq)$. A semigroup $S$ is called a {\bf
complete intersection} if the semigroup ring $\Bbbk[[t^S]]$ is
complete intersection, or equivalently if the
 cardinality of any of its minimal presentations equals $\nu(S)-1$ (cf. \cite{RG}, page 129).

Numerical semigroups other than $ \mathbb{N}$ are never unique
factorization monoids, as there are always elements with different
decompositions into irreducibles (note that in our context an
irreducible element is the same thing as a minimal generator). If
$ s= \lambda_1 g_1 + \cdots + \lambda_\nu g_\nu$ with $ \lambda_i
\in \mathbb{N}$ we say that $ \lambda_1 g_1 + \cdots + \lambda_\nu
g_\nu$ is a {\bf representation} of $ s $. Given $ s \in S $, we
define the {\bf $ M$-adic order} as $ \ord(s) = \max
\{\sum_{i=1}^{ \nu } \lambda_i \, |  \sum_{i=1}^{ \nu } \lambda_i
g_i \text{ is a representation of }s\}$. We say that $ s =
\lambda_1 g_1 + \cdots + \lambda_\nu g_\nu$ is a {\bf maximal
representation} of $ s $ if $ \sum_{i=1}^{\nu } \lambda_i = \ord
(s) $. We can define an other partial order on $ S $ setting $ s
\preceq_M t $ if there exists $ u \in S $ such that $ s + u = t $
and $ \ord(s) + \ord(u) = \ord(t) $ (cf. \cite{Br}). The number of
representations and of maximal representations of elements in a
semigroup is related to some of the objects of our study; see
\cite{BCKR} for more on factorization in numerical semigroups.

The book  \cite{RG} is an exhaustive source on the subject of
numerical semigroups.
\\

 We { now} give the main definitions of the paper.

\begin{defins}\label{definIntegers}
Let $S$ be a numerical semigroup minimally generated by $g_1 <
\cdots < g_{\nu}$. For each $i=2, \ldots, \nu$ define:
\begin{eqnarray*}
\tau_i &=& \tau_i(S)=\min\{ h \in \mathbb{N} \, | \, hg_i \in \langle g_1, \ldots, g_{i-1} \rangle \}-1; \\
\alpha_i &=& \alpha_i(S)= \max\{ h \in \mathbb{N} \, | \, hg_i \in \ap(S)\}; \\
\beta_i &=& \beta_i(S)=\max\{ h \in \mathbb{N} \, | \, hg_i \in \ap(S) \text{ and } \ord(h g_i)=h\};\\
\gamma_i &=& \gamma_i(S)=\max\{ h \in \mathbb{N} \, | \, hg_i \in
\ap(S), \, \ord(hg_i)=h \text{ and} \\
& &  \hspace*{3.8cm} hg_i \text{ has a unique maximal
representation}\}.
\end{eqnarray*}

If $\mathbf{n}=\{ n_1, \ldots, n_\nu\}$ is any rearrangement of the minimal generators (i.e.,
the minimal system of generators not necessarily in increasing order),
define for each $i=2, \ldots, \nu$:\\

$
\phi_i \hspace*{0.2cm} = \hspace*{0.2cm}\phi_i(S,\mathbf{n})=
\min\{ h \in \mathbb{N} \, | \, hn_i \in \langle n_1, \ldots, n_{i-1} \rangle \}-1.$
\end{defins}

\begin{rem}\label{incl}
For each index $i=2, \ldots, \nu,$ we clearly have $\gamma_i \leq
\beta_i \leq \alpha_i$. This, together with the fact that $\ap(S)
\subseteq \Big\{\sum_{i=2}^{\nu} \lambda_i g_i \, | \, 0 \leq
\lambda_i \leq \gamma_i \Big\}$ (cf. \cite[Corollary 2.7]{DMS2}),
implies that
$$
\ap(S) \subseteq \Big\{\sum_{i=2}^{\nu} \lambda_i g_i \, | \, 0 \leq \lambda_i \leq \gamma_i \Big\}
\subseteq  \Big\{\sum_{i=2}^{\nu} \lambda_i g_i \, | \, 0 \leq \lambda_i \leq \beta_i \Big\}
\subseteq \Big\{\sum_{i=2}^{\nu} \lambda_i g_i \, | \, 0 \leq \lambda_i \leq \alpha_i \Big\}.
$$
In particular, we have
$ m=|\ap(S)|\leq \prod_{i=2}^{\nu}(\gamma_i+1)
\leq \prod_{i=2}^{\nu} (\beta_i+1) \leq \prod_{i=2}^{\nu} (\alpha_i+1).$
\end{rem}

\begin{defins}\label{defclasses}
Let $S$ be a numerical semigroup
minimally generated by $g_1 <
\cdots < g_{\nu}$.
\begin{enumerate}
\item $S$ is {\bf telescopic} if
$\ap(S)=\Big\{\sum_{i=2}^{\nu} \lambda_i g_i \, | \, 0 \leq
\lambda_i \leq \tau_i \Big\}$;

\item $S$ is {\bf associated to a plane branch}
if $S$ is telescopic and $(\tau_i+1) g_i < g_{i+1}$ for all $i=2,\ldots,\nu-1$;

\item $S$ has {\bf $\alpha$-rectangular Ap\'ery set} if
$\ap(S)= \Big\{\sum_{i=2}^{\nu} \lambda_i g_i \, | \, 0 \leq \lambda_i \leq \alpha_i \Big\}$;

\item $S$ has {\bf $\beta$-rectangular Ap\'ery set} if
$\ap(S)= \Big\{\sum_{i=2}^{\nu} \lambda_i g_i \, | \, 0 \leq \lambda_i \leq \beta_i \Big\}$;

\item $S$ has {\bf $\gamma$-rectangular Ap\'ery set} if
$\ap(S)= \Big\{\sum_{i=2}^{\nu} \lambda_i g_i \, | \, 0 \leq \lambda_i \leq \gamma_i \Big\}$;

\item $S$ is {\bf free} if there exists a rearrangement $\mathbf{n}=\{ n_1, \ldots, n_\nu\}$
of the minimal generators such that
$\ap(S,n_1)=\Big\{\sum_{i=2}^{\nu} \lambda_i n_i \, | \, 0 \leq \lambda_i \leq \phi_i \Big\}$.

\end{enumerate}
\end{defins}

{Notice that the definitions of telescopic and free semigroups
are not standard, but it is proved in \cite{RG} that the
conditions we state are equivalent to the classical definitions.}

We turn now to the study of semigroups with $\alpha$-rectangular
Ap\'ery set providing some characterizations, then we collect analogous statements for classes (1), (4), (5) and (6).
In \cite{Ro} Rosales introduced the following definition: a
numerical semigroup $S$ has {\bf  Ap\'ery set of unique
expression} if every element in $\ap(S)$ has a unique
representation. We will see that this condition is closely related
to having $\alpha$-rectangular Ap\'ery set.

\begin{lem}[\cite{FGH}, Lemma 6]\label{easylemma}
If $s \preceq t$ and $t\in \ap(S)$, then $s\in \ap(S)$.
\end{lem}

\begin{lem}\label{4,5}
If $s \preceq t$ and $t$ has a unique representation, then $s
\preceq_M  t$ and $s$ has a unique representation.
\end{lem}

\begin{proof}
If an element has a unique representation
then this must be maximal and
the sum of the coefficients equals the order of the element.
Let $t=\sum_{i=1}^{\nu} \lambda_i g_i$ and $s+u= t$ for some $u \in S$.
Since the representation of $t$ is unique, it follows that $s=\sum_{i=1}^{\nu} \xi_i g_i$ and
$u=\sum_{i=1}^{\nu} \rho_i g_i$, with $\rho_i + \xi_i = \lambda_i$ for each $i$.
These representations must be unique,
otherwise $t$ has a double representation,
 and we get $\ord(s)+\ord(u)=\sum_{i=1}^{\nu} \xi_i +\sum_{i=1}^{\nu} \rho_i= \sum_{i=1}^{\nu} \lambda_i= \ord(t)$.
\end{proof}

\begin{prop}\label{charalpha}
The following conditions are equivalent:

\begin{enumerate}
     \item[(i)]
      $\ap(S)$ is $\alpha$-rectangular;
      \item[(ii)]
      there is only one maximal element in $(\ap(S), \preceq)$ and it has a unique representation;
       \item[(iii)] $S$ is symmetric and  $\ap(S)$ is of unique expression;
       \item[(iv)] $f+m=\sum_{i=2}^{\nu} \alpha_i g_i$;
       \item[(v)] $m=\prod_{i=2}^{\nu}(\alpha_i+1)$.
\end{enumerate}
\end{prop}
\begin{proof}

$(i) \Rightarrow (ii)$ Since $\ap(S)$ is $\alpha$-rectangular, we
immediately get that $\sum_{i=2}^{\nu} \alpha_i g_i$ is the unique
maximal element in $(\ap(S), \preceq)$. Let us suppose that
$\sum_{i=2}^{\nu} \alpha_i g_i = \sum_{i=2}^{\nu} u_i g_i$ for
some non-negative integers $u_i$.
By Lemma \ref{easylemma}, $u_i g_i \in \ap(S)$ for
each $i$ and hence $u_i \leq \alpha_i$, by definition of
$\alpha_i$; it follows that $u_i=\alpha_i$ for each index $i$ and
the two representations coincide.

$(ii) \Leftrightarrow (iii)$ It follows by Lemma \ref{4,5} and by
the fact that $S$ is symmetric if and only if $f+m$ is the only maximal element of $(\ap(S), \preceq)$.

$(ii) \Rightarrow (iv)$ The unique maximal element in $(\ap(S),
\preceq)$ is necessarily $f+m$.
Therefore $\alpha_i g_i \preceq f+m$
for each $i=2,\ldots, \nu$. Since $f+m$ has a unique
representation, the thesis follows immediately.

$(iv) \Rightarrow (i)$ Since $f+m \in \ap(S)$ in general, it follows by
Lemma \ref{easylemma}.

$(i) \Rightarrow (v)$ It follows by $m=|\ap(S)|$ and by the fact that $\ap(S)$ is of unique expression.

$(v) \Rightarrow (i)$ We already noticed that $\ap(S) \subseteq
\Big\{\sum_{i=2}^{\nu} \lambda_i g_i \, | \, 0 \leq \lambda_i \leq
\alpha_i \Big\}$ and since
$m=\prod_{i=2}^{\nu}(\alpha_i+1)=|\ap(S)|$, we must have an
equality.
\end{proof}

\begin{ex}
We  apply the criterion above to show that the Ap\'ery set of
$S= \langle 12,15,16,18 \rangle$ is $\alpha$-rectangular, without even
computing the whole  $\ap(S)$. We determine the
$\alpha_i$'s:
\begin{eqnarray*}
2\cdot 15 &=& 12+ 18 \in 12 + S \\
2 \cdot 16 &=& 32 \notin 12+S\\
3 \cdot 16 &=& 4 \cdot 12 \in 12 +S\\
2 \cdot 18 &=& 3 \cdot 12 \in 12 + S
\end{eqnarray*}
and so $\alpha_2=1,\, \alpha_3=2,\, \alpha_4=1 $ and $m=12=
2\cdot2\cdot3 =\prod_{i=2}^{\nu}(\alpha_i+1)$.
\end{ex}

A
semigroup is called {\bf M-pure} if all the maximal elements in
the poset $(\ap(S), \preceq_M)$ have the same order; $M$-pure
semigroups were introduced in \cite{Br} along the way to the
characterization of  Gorenstein associated graded rings.
In analogy to \cite{Ro}, we say that a semigroup $S$ has {\bf
Ap\'ery set of unique maximal expression} if every element in
$\ap(S)$ has a unique maximal representation. In connection to
this, the number of maximal representations of elements in a
semigroup has been investigated recently (cf. \cite{BH},
\cite{BHJ}).
Now we give the criteria for the remaining classes. 

\begin{prop}[\cite{DMS2}, Theorem 2.16]\label{charactbeta}
The following conditions are equivalent:

\begin{enumerate}
     \item[(i)]
      $\ap(S)$ is $\beta$-rectangular;
       \item[(ii)] $S$ is $M$-pure, symmetric and  $\ap(S)$ is of unique maximal expression;
       \item[(iii)] $\ap(S)$ has a unique maximal element with respect to $\preceq_M$ and this element has a
       unique maximal representation;
       \item[(iv)] $f+m=\sum_{i=2}^{\nu} \beta_i g_i$;
       \item[(v)] $m=\prod_{i=2}^{\nu}(\beta_i+1)$.
\end{enumerate}
\end{prop}

\begin{prop}[\cite{DMS2}, Theorem 2.22]\label{charact}
The following conditions are equivalent:
\begin{enumerate}
     \item[(i)]
      $\ap(S)$ is $\gamma$-rectangular;

     \item[(ii)] $f+m=\sum_{i=2}^{\nu} \gamma_i g_i$;

\item[(iii)] $m=\prod_{i=2}^{\nu}(\gamma_i+1)$.

\end{enumerate}
\end{prop}

\begin{prop}[\cite{RG}, Proposition 9.15]\label{chartel}
The following conditions are equivalent:
\begin{enumerate}
     \item[(i)]
      $S$ is telescopic;

     \item[(ii)] $f+m=\sum_{i=2}^{\nu}\tau_ig_i$;

     \item[(iii)] $m=\prod_{i=2}^{\nu} (\tau_i+1)$.
\end{enumerate}
\end{prop}

\begin{prop}[\cite{RG}, Proposition 9.15]\label{charfree}
The following conditions are equivalent:
\begin{enumerate}
     \item[(i)]
      $S$ is free;

        \item[(ii)] there is an arrangement $\mathbf{n}$ of the minimal
        generators such that $f+n_1=\sum_{i=2}^{\nu}\phi_in_i$;

        \item[(iii)] there is an arrangement $\mathbf{n}$ of the minimal
        generators such that $n_1=\prod_{i=2}^{\nu} (\phi_i+1)$.
\end{enumerate}
\end{prop}

The next lemma is a crucial step in establishing one of the  implications in Theorem \ref{main}.

\begin{lem}\label{lemmaSigma}
Let $S$ have $\gamma$-rectangular Ap\'ery set.
For each  $i=2,\dots,\nu$ there exist relations
\begin{equation}\label{eq1}\tag{$\star $}
(\gamma_i+1)g_i=\lambda_{i,1}g_1+\lambda_{i,2}g_2+\cdots+\lambda_{i,\nu}g_{\nu}
\end{equation}
and a permutation
$\sigma$ of  $\{1,\ldots,\nu\}$
 such that $\sigma(1)=1$ and $\lambda_{\sigma(i),\sigma(j)}=0$ if  $i\leq j$, $j \geq 2$.
\end{lem}

\begin{proof}
Fix  an index $i \in \{2, \ldots,\nu\}$. By definition of
$\gamma_i$ we have two possible cases:
\begin{itemize}
\item[(I)] If $(\gamma_i+1)g_i\in \ap(S)$,
then the representation $(\gamma_i+1)g_i$ is not { maximal or
it is not the} unique maximal one; hence there is a different
representation $ \sum_{j=1}^{\nu} \lambda_{i,j}g_j $ of the same
element with $\gamma_i +1 \leq \sum_{j=1}^{\nu} \lambda_{i,j}$.
Notice that $(\gamma_i+1)g_i\in \ap(S)$ forces $\lambda_{i,1}=0$.

\item[(II)]  If $(\gamma_i+1)g_i\notin \ap(S)$,
then
we can write $(\gamma_i+1)g_i = \sum_{j=1}^{\nu} \lambda_{i,j}g_j$ for some non-negative integers $\lambda_{i,j}$,
with $\lambda_{i,1}>0$.
\end{itemize}
It is useful to consider the   square matrix
obtained from the relations (\ref{eq1}) found in (I) and (II) leaving out the coefficients of $g_1$
$$
{\bf L}=\begin{pmatrix} \lambda_{2,2}&\lambda_{2,3}&\dots&\lambda_{2,\nu}\\
\lambda_{3,2}&\lambda_{3,3}&\dots&\lambda_{3,\nu}\\
\dots&\dots&\dots&\dots\\
\lambda_{\nu,2}&\lambda_{\nu,3}&\dots&\lambda_{\nu,\nu}\end{pmatrix}.
$$

Now we construct a permutation $\sigma$ of $\{1,2,\dots,\nu\}$ satisfying
$\sigma(1)=1$ and  $\lambda_{\sigma(i),\sigma(j)}=0$ whenever  $i\le j$ and $j \geq 2$,
or equivalently such that
the square matrix
$$
{\bf L}_\sigma=\begin{pmatrix} \lambda_{\sigma(2),\sigma(2)}&\lambda_{\sigma(2),\sigma(3)}&\dots&\lambda_{\sigma(2),\sigma(\nu)}\\
\lambda_{\sigma(3),\sigma(2)}&\lambda_{\sigma(3),\sigma(3)}&\dots&\lambda_{\sigma(3),\sigma(\nu)}\\
\dots&\dots&\dots&\dots\\
\lambda_{\sigma(\nu),\sigma(2)}&\lambda_{\sigma(\nu),\sigma(3)}&\dots&\lambda_{\sigma(\nu),\sigma(\nu)}\end{pmatrix}
$$
 is lower triangular with zeros in the diagonal. 
 We proceed by decreasing induction on $h$.
 
For the basis of the induction   $h=\nu$ it is enough to show that there exists a column in $\bf L$ with all zero entries.
Let us suppose by contradiction that every column in $\bf L$ has a non zero element, that is,
for every $j\ge 2$ there exists $\tau(j)$ such that $\lambda_{\tau(j),j}>0$.
Taking the sum over all the relations (\ref{eq1}) we obtain
$$\sum_{i=2}^{\nu}(\gamma_i+1)g_i=\sum_{i=2}^{\nu}\sum_{j=1}^{\nu} \lambda_{i,j}g_j.$$
and subtracting $\sum_{i=2}^{\nu}g_i$ from both sides  we get
 $$u:=\sum_{i=2}^{\nu}\gamma_ig_i=\sum_{j=1}^{\nu}\sum_{i\ne 1,\tau(j)}^{\nu} \lambda_{i,j}g_j+\sum_{i=2}^{\nu}( \lambda_{\tau(i),i}-1)g_i.$$
 As $u\in\ap(S)$
by $\gamma$-rectangularity,
we necessarily have
$  \lambda_{i,1}=0$  and hence case (II) above is not possible for any $i\in \{2,\dots,\nu\}$.
We get by (I) that $\sum_{j=1}^{\nu} \lambda_{i,j} \geq \gamma_i+1$ for every $i$.
Furthermore,
the representation $u=\sum_{i=2}^{\nu}\gamma_ig_i$ is maximal by \cite[Lemma 2.19]{DMS2} and so if there exists $i$ such that  $\sum_{j=1}^{\nu} \lambda_{i,j}>\gamma_i+1$
then it follows
 $$\ord(u)=\sum_{i=2}^{\nu}\gamma_i< \sum_{j=1}^{\nu}\sum_{i\ne 1,\tau(j)}^{\nu} \lambda_{i,j}+\sum_{i=2}^{\nu}( \lambda_{\tau(j),j}-1)\leq \ord(u)$$
 yielding a contradiction;
thus  $\sum_{j=1}^{\nu} \lambda_{i,j}=\gamma_i+1$ for every $i$.
In particular for the index of the largest generator we have
$$
(\gamma_{\nu}+1)g_{\nu}=\sum_{j=1}^{\nu} \lambda_{\nu,j}g_j \quad
 \mbox{ and  }\quad \sum_{j=1}^{\nu} \lambda_{\nu,j}=\gamma_{\nu}+1.$$
 But  $g_j<g_{\nu}$ for $j\ne\nu$ forces $\lambda_{\nu,j}=0$ and  $\lambda_{\nu,\nu}=\gamma_{\nu}+1$, contradicting
 the fact that in (I) we found a different representation.
So the $p$-th column of $\bf L$ consists of zeros for some $p\geq 2$,
and we let  $\sigma(\nu)=p$.

Now let $1<h<\nu$ and suppose that  for every
$j\in\{ \sigma(\nu),\sigma(\nu-1),\dots,\sigma(h+1)\}$
and  $i\le j$ we have $\lambda_{\sigma(i),\sigma(j)}=0$. By
repeating the same argument as in the basis of the induction for
the  submatrix of $\bf L$ indexed by $i,j\in
\{2,\dots,\nu\}\setminus\{  \sigma(\nu),\sigma(\nu-1),\dots,\sigma(h+1)\}
$ we get a new index  $\sigma(h)$ for which the statement is
true, and the inductive step follows.
\end{proof}

In order to present the main theorem of the paper,
we need to give one more definition.
A numerical semigroup $S$ has a {\bf unique Betti element} if the first syzygies of the semigroup ring $\Bbbk[[t^S]]$
have all the same degree (in the $S$-grading; see \cite{GOR} for a purely numerical definition).
In \cite{GOR} the authors prove that   $S=\langle g_1,\ldots,g_\nu \rangle$ has a unique Betti element if and only if there exist pairwise coprime integers  $a_1, \ldots, a_\nu$ greater than one  such that $g_i = \prod_{j\ne i} a_i$;
 these semigroups are shown to be complete intersection.
Moreover
in \cite{BF}
it is shown that for such a semigroup $S$ the tangent cone of the semigroup ring $\Bbbk[[t^S]]$ is a complete intersection,
implying thus that $\ap(S)$ is $\gamma$-rectangular by \cite[Theorem 3.6]{DMS2}.

 \begin{thm}\label{main}
 Let $S$ be a numerical semigroup. Consider the following conditions:
\begin{enumerate}
  \item[(1)] $S$ is associated to a plane branch;

   \item[(2)] $S$ has a unique Betti element;

     \item[(3)] $S$ is telescopic;

     \item[(4)] $S$ has $\alpha$-rectangular Ap\'ery set;

     \item[(5)] $S$ has $\beta$-rectangular Ap\'ery set;

     \item[(6)] $S$ has $\gamma$-rectangular Ap\'ery set;

     \item[(7)] $S$ is free;

     \item[(8)] $S$ is complete intersection.
\end{enumerate}
Then  $(1)\Rightarrow(3)\Rightarrow (5) $,
$(2)\Rightarrow(4)\Rightarrow (5)$,
 $(1)\Rightarrow(4)$,
$(2)\Rightarrow(3)$,
   $(5)\Rightarrow (6)\Rightarrow(7) \Rightarrow(8)$
   (compare Figure \ref{diagram}).
Moreover, all the implications are strict.
\end{thm}
\begin{proof}
In each of the proofs below, let $S$ be
minimally generated by $g_1 <
\cdots < g_{\nu}$.
\\
$\bullet$
Plane branch $\Rightarrow $ Telescopic.
\\
 It follows from Definitions \ref{defclasses}.
The semigroup $S= \langle 6,10,15 \rangle$ is not associated to a plane branch,
as $(\tau_2+1)g_2=3\cdot 10 >15 = g_3$;
however $S$ has a unique Betti element, in particular it is telescopic and with $\alpha$-rectangular Ap\'ery set (see below).
\\
$\bullet$
Plane branch $\Rightarrow \alpha$-rectangular Ap\'ery set.
\\
We prove  that $(\tau_i+1)g_i\notin \ap(S)$ by induction on $i\in\{2,\ldots,\nu\}$.
Since $(\tau_2+1)g_2\in \langle g_1 \rangle$ we get $(\tau_2+1)g_2 \notin \ap(S)$.
Given $i>2$, we have $(\tau_i+1)g_i=\lambda_1g_1+\cdots + \lambda_{i-1} g_{i-1}$ for some $\lambda_j \in \mathbb{N}$.
Assume by contradiction $(\tau_i+1)g_i\in \ap(S)$,
then by induction and Lemma \ref{easylemma} we must have $\lambda_1=0$ and $\lambda_j\leq \tau_j$ for $j=2,\ldots,i-1$.
By definition  of semigroup associated to a plane branch, we have the following chain of inequalities:
\begin{eqnarray*}
(\tau_i+1)g_i&\geq& 2g_i > 2(\tau_{i-1}+1) g_{i-1}\geq(\tau_{i-1}+1) g_{i-1}+2g_{i-1}>\\
&>& (\tau_{i-1}+1) g_{i-1}+2(\tau_{i-2}+1)g_{i-2}\geq \cdots \geq\\
&\geq& (\tau_{i-1}+1)g_{i-1}+\cdots + (\tau_{2}+1) g_{2} >\lambda_1g_1+\cdots + \lambda_{i-1} g_{i-1} = (\tau_i+1)g_i
\end{eqnarray*}
reaching a contradiction.
Hence  $(\tau_i+1)g_i\notin \ap(S)$ and $\alpha_i \leq \tau_i$.
Finally $\ap(S)$ is $\alpha$-rectangular by Proposition \ref{charalpha} $(v)$ as
$
m \leq \prod_{i=2}^{\nu} (\alpha_i+1) \leq \prod_{i=2}^{\nu} (\tau_i+1) =m,
$
where we used Remark \ref{incl} and the fact that $S$ is telescopic.
\\
$\bullet$
Unique Betti element $\Rightarrow \alpha$-rectangular Ap\'ery set.
\\
Let $a_1>a_2>\cdots >a_\nu>1$ be pairwise coprime integers such
that $g_i = \prod_{j\ne i} a_j$. Similarly to the previous proof,
it suffices to show that $\alpha_i+1\leq a_i$. But this is trivial
as $a_ig_i=a_1g_1 \notin \ap(S)$. Now let $S=\langle
4,6,13\rangle$: we have $\ap(S)=\{0,6,13,19\}$, $\tau_2=\tau_3=1$,
and $m=(\tau_2+1)(\tau_3+1)$, $(\tau_2+1)g_2 < g_3$. So $S$ is
associated to a plane branch and hence telescopic and with
$\alpha$-rectangular Ap\'ery set, but $S$  does not have a unique
Betti element.
\\
$\bullet$
Unique Betti element $\Rightarrow$ Telescopic.
\\
Let $a_1>a_2>\cdots >a_\nu>1$ be pairwise coprime integers such that $g_i =
\prod_{j\ne i} a_j$. We show that $\tau_i = a_i-1$ for each $i
\geq 2$, from which it follows that $S$ is telescopic by
Proposition \ref{chartel} $(iii)$. Since the $a_j$'s are coprime,
$a_i$ does not divide $h g_i$ for $h\leq a_i-1$, hence $hg_i
\notin \langle g_1, \ldots, g_{i-1}\rangle$. However $a_i g_i
=a_1g_1 \in \langle g_1, \ldots, g_{i-1}\rangle$ so that
$\tau_i=a_i-1$.
\\
$\bullet$ $\alpha$-rectangular Ap\'ery set $\Rightarrow$
$\beta$-rectangular Ap\'ery set $ \Rightarrow $
$\gamma$-rectangular Ap\'ery set.
\\
It follows from Remark \ref{incl}.
The semigroup  $S=\langle 8,10,15\rangle$ is telescopic and therefore $\ap(S)$ is $\beta$-rectangular (see below),
but it is not $\alpha$-rectangular:
 $\ap(S)=\{ 0,  10, 15, 20, 25, 30, 35, 45 \}$ and it is easy to check that
$\alpha_2=\alpha_3=3$ and $\tau_2=3, \tau_3=1$ so that
$m=(\tau_2+1)(\tau_3+1)$ but $m\ne(\alpha_2+1)(\alpha_3+1)$.
The  Ap\'ery set
of $S=\langle 8,10,11,12\rangle$ is $\gamma$-rectangular but not $\beta$-rectangular:
we have $\ap(S)=\{ 0,  10, 11, 12, 21, 22, 23, 33 \}$ and we get $\beta_2=1, \beta_3=3, \beta_4=1, \gamma_2= \gamma_3= \gamma_4=1$,
hence
$m=(\gamma_2+1)(\gamma_3+1)(\gamma_4+1)$ and $m\ne(\beta_2+1)(\beta_3+1)(\beta_4+1)$.
\\
$\bullet$ Telescopic $ \Rightarrow $ $\beta$-rectangular Ap\'ery set.
\\
For each $i\in\{2,\ldots,\nu\}$ we have $(\tau_i+1)g_i =
\lambda_1g_1+\cdots + \lambda_{i-1} g_{i-1}$ for some $\lambda_j
\in \mathbb{N}$. The fact that $g_1<\cdots<g_{i-1}<g_i$ forces
$\lambda_1+\cdots + \lambda_{i-1}>\tau_i+1$ and therefore
$\ord((\tau_i+1)g_i)>\tau_i+1$. It follows that $\beta_i \leq
\tau_i$ and by Remark \ref{incl} we get $ m\leq \prod_{i=2}^\nu
(\beta_i+1)\leq  \prod_{i=2}^\nu (\tau_i+1) =m $ and hence
$\ap(S)$ is $\beta$-rectangular by Proposition \ref{charactbeta}
$(v)$. Let $S= \langle 4, 5, 6 \rangle$: we have
$\ap(S)=\{0,5,6,11\}$ and thus $\alpha_2=\alpha_3=1$, $\tau_2=3,
\tau_3=1$ so that $\ap(S)$ is $\alpha$-rectangular as
$m=(\alpha_2+1)(\alpha_3+1)$ (hence $\beta$-rectangular) but $S$
is not telescopic as $m\ne(\tau_2+1)(\tau_3+1)$.
\\
$\bullet$ $\gamma$-rectangular Ap\'ery set $\Rightarrow$ Free.
\\
Assume  $S$ has $\gamma$-rectangular Ap\'ery set. Let $\sigma $ be
the permutation of $\{1,\ldots,\nu\}$ as in Lemma
\ref{lemmaSigma}, and consider the rearrangement of the minimal
generators ${\bf n}= \{ n_1, \ldots, n_\nu\}$ with
$n_i=g_{\sigma(i)}$. By relations (\ref{eq1})  for each
$i=2,\ldots,\nu$  we get
$$(\gamma_{\sigma(i)}+1)n_i=(\gamma_{\sigma(i)}+1)g_{\sigma(i)}=\sum_{j=1}^{\nu}\lambda_{\sigma(i),j}g_j=\sum_{j=1}^{\nu}\lambda_{\sigma(i),\sigma(j)}g_{\sigma(j)}=\sum_{j=1}^{\nu}\lambda_{\sigma(i),\sigma(j)}n_j$$
thus $\phi_i\le\gamma_{\sigma(i)}$ by the triangularity of the matrix ${\bf L}_{\sigma}$.
Following the notation  of \cite{RG}, let
$$\overline{c}_i=\min\big\{h\in\mathbb{N}\setminus\{0\}\ \big|\ \gcd(n_1,\dots,n_{i-1}) \mbox{ divides } hn_i\big\}.$$
In \cite[Lemma 9.13]{RG} it is proved that $n_1=\prod_{i=2}^{\nu}\overline{c}_i$ and $\overline{c}_i\le \phi_i+1$.
On the other hand $n_1=\prod_{i=2}^{\nu}(\gamma_{i}+1)=\prod_{i=2}^{\nu}(\gamma_{\sigma(i)}+1)$ by Proposition \ref{charact} $(iii)$.
We conclude that
$$
n_1=\prod_{i=2}^{\nu}\overline{c}_i\le \prod_{i=2}^{\nu}(\phi_i+1)\leq \prod_{i=2}^{\nu}(\gamma_{\sigma(i)}+1)=n_1
$$
hence $n_1=\prod_{i=2}^{\nu}(\phi_{i}+1)$ and  $S$ is free by Proposition \ref{charfree} $(iii)$.

Let $S=\langle 5,6,9 \rangle$.
Since $5$ is prime,
we cannot have $m=(\gamma_2+1)(\gamma_3+1)$,
therefore $\ap(S) $ is not $\gamma$-rectangular.
Consider the arrangement
${\bf n}=\{6,9,5\}$:
we have
$\phi_2=1,\phi_3=2$ so that $S$ is free as $n_1=(\phi_2+1)(\phi_3+1)$.
\\
$\bullet$ Free $\Rightarrow$ Complete intersection.
\\
This is well-known and is proven e.g. in \cite[Corollary 9.17]{RG} by means of gluing.
Counterexamples for the inverse implication are provided at the beginning of the next section.
\end{proof}

\section{Gluing and other applications}

In this section we explore
an operation that allows to construct new (more complicated) semigroups from old ones.
Let $S_1$ and $S_2$ be two numerical semigroups minimally generated by $n_1,\ldots,n_r$
and $m_{1},\ldots, m_s$, respectively.
Given positive integers $d_1 \in S_1 \setminus\{n_1,\ldots, n_r\}$
and $d_2 \in S_2 \setminus\{m_{1},\ldots, m_s\}$ such that
$\gcd(d_1,d_2)= 1$,
the semigroup
$$
S =d_2 S_1+d_1S_2= \langle d_2n_1,\ldots , d_2n_r, d_1 m_{1},\ldots , d_1 m_s\rangle
$$
is called a {\bf gluing} of $S_1$ and $S_2$.
Notice that $\nu(S)=\nu(S_1)+\nu(S_2)$.
The importance of gluing was first highlighted in \cite{De},
where the author proved that a semigroup is a complete intersection
if and only if it is a gluing of two complete intersection semigroups,
formulating thus a recursive characterization.
A gluing of two  symmetric semigroups is again symmetric.
Although the gluing of two free semigroups needs not be free,
a semigroup of embedding dimension $\nu$ is free if and only if it is a gluing of $\mathbb{N}$
and a free semigroup of embedding dimension $\nu-1$ (cf. \cite[Theorem 9.16]{RG}).
We remark that gluing has other interesting applications,
e.g. to Rossi's conjecture (cf. \cite{AM}, \cite{JZ}) and to Huneke-Wiegand conjecture (cf. \cite{GL}).

\begin{ex}
As an illustration,
we construct a family of complete intersection semigroups that are not free.
Let $p_1, p_2,p_3,p_4$ be distinct primes such that $p_3, p_4 > p_1p_2$.
Consider
$$ S=\langle p_1p_3,p_2p_3,p_1p_4,p_2p_4 \rangle = d_2T+d_1T$$
where $T=\langle p_1,p_2\rangle$, $d_1=p_4$ and $d_2=p_3$ (note
$p_3,p_4\in T\setminus\{p_1,p_2\}$ as $f(T)=p_1p_2-p_1-p_2$). Now
$T$ is a complete intersection being two-generated, therefore $S$
is a complete intersection. However, there is no hope of
expressing $S$ as a gluing of $\mathbb{N}$ and a three-generated
semigroup because any three generators of $S$ are coprime; by the
characterization above $S$ is not free.
\end{ex}

\begin{rem}
By \cite[Theorem 9.16]{RG} and by definition, it is easy to see
that a semigroup $S$ is telescopic if and only if it is a gluing
of $\mathbb{N}$ and a telescopic semigroup $T=\langle n_1,\dots ,
n_{\nu-1}\rangle$ with  $d_2 > d_1n_{\nu-1}$.

Furthermore, it is also easy to check that a semigroup $S=\langle
g_1, \ldots, g_\nu\rangle $ has a unique Betti element if and only
if it is the gluing  $d_1 T + d_2\mathbb{N}$ where $T=\langle n_1,
\ldots, n_{\nu-1}\rangle$ has a unique Betti element,
 $d_2= \mathrm{lcm}(n_1, \ldots, n_{\nu-1})$ and $\gcd(n_i, d_1)=1$ for each $i$.

Finally, by definition, a semigroup is associated to a plane
branch if and only if it is a gluing of $\mathbb{N}$ and a
semigroup associated to a plane branch $T=\langle n_1,\dots ,
n_{\nu-1}\rangle$ with  $d_2
> d_1(\tau_{\nu-1}(T)+1)n_{\nu-1}$.
\end{rem}

Our aim at this point is to push this study further:
we use gluing to prove a recursive characterization for  semigroups with $\alpha$-rectangular Ap\'ery sets.

\begin{thm}\label{TheoremGlueAlpha}
Let $T$ be a  semigroup with $\alpha$-rectangular Ap\'ery set and
$d_1, d_2 \in \mathbb{N}$ such that $d_1 \notin \ap(T)$, $d_1> d_2m(T)$;
then the gluing $S=d_2T+d_1\mathbb{N}$  has $\alpha$-rectangular Ap\'ery set.
Conversely, every  semigroup $S\ne \mathbb{N}$ with $\alpha$-rectangular Ap\'ery set arises in this way.
\end{thm}

\begin{proof}
Assume that $S$ is the gluing $d_2T+d_1\mathbb{N}$  where  $T=\langle n_1<\cdots<n_{\nu-1}\rangle$ has $\alpha$-rectangular Ap\'ery set and  $d_1 \in T\setminus\{n_1, \ldots, n_{\nu-1}\},d_2\in \mathbb{N}\setminus\{1\} $ are coprime integers
satisfying
 $d_1 \notin \ap(T)$ and $d_1> d_2m(T)$;
 in particular we have
$m(S)=d_2m(T)$.
In the proof of this implication
 $\alpha_i(S)$ denotes, with an abuse of notation, the integer $\alpha$ from Defintions \ref{definIntegers} relative to the  minimal generator $d_2n_i$ of $S$
(which is not necessarily the $i$-th generator of $S$ in increasing order).
By Proposition \ref{charalpha} $(v)$, $n_1 = \prod_{i=2}^{\nu-1}(\alpha_i(T)+1)$.
We claim that $\alpha_i(S)\leq \alpha_i(T)$ for each  $i=2,\ldots, \nu-1$.
In fact
$$
(\alpha_i(T)+1)n_i=\lambda_1n_1+\cdots+ \lambda_{\nu-1} n_{\nu-1}
\, \Longrightarrow \,
(\alpha_i(T)+1)d_2n_i=\lambda_1 d_2n_1+\cdots+ \lambda_{\nu-1}d_2  n_{\nu-1}
$$
for some $\lambda_j \in \mathbb{N}$ with $\lambda_1>0$.
Since $m(S)=d_2n_1$ we get
$(\alpha_i(T)+1)d_2n_i\notin \ap(S)$,
proving that $\alpha_i(S)\leq \alpha_i(T)$.
Now we show that $ \alpha_\nu(S) \leq d_2-1$:
we have
$d_1-n_1\in T$ as $d_1\notin \ap(T)$,
therefore
$d_2d_1 - d_2 n_1 \in S$
and $d_2 d_1 \notin \ap(S)$.
By Remark \ref{incl}
$$
m(S) \leq \prod_{i=2}^\nu (\alpha_i(S)+1) \leq d_2 \prod_{i=2}^{\nu-1} (\alpha_i(T)+1)=d_2n_1=m(S)
$$
and hence $\ap(S)$ is $\alpha$-rectangular by $m(S)= \prod_{i=2}^\nu (\alpha_i(S)+1)$.
\\

Assume now that $S=\langle g_1<\cdots<g_\nu\rangle \ne \mathbb{N}$ has $\alpha$-rectangular Ap\'ery set.
By Theorem \ref{main} $\ap(S)$ is $\gamma$-rectangular and thus
there is a rearrangement ${\bf n} = \{ n_1, \ldots, n_\nu\}$  of the minimal
generators such that $g_1=n_1$ and fulfilling the conditions of Proposition \ref{charfree};
let $\sigma$ be the permutation such that $n_i=g_{\sigma(i)}$.
Let $d=\gcd(n_1,\ldots, n_{\nu-1})$.
Then $S$ is the gluing of
$T=\Big\langle \frac{n_1}{d},\ldots, \frac{n_{\nu-1}}{d}\Big\rangle$ and $\mathbb{N}$,
with integers $d_1=n_\nu$ and $d_2=d$;
furthermore $T$ is free by  \cite[Theorem 9.16]{RG}.
We prove that $\ap(T)$ is $\alpha$-rectangular.

Let $l= \sigma(\nu)$; it is shown in \cite[Lemma  9.13 (3), Proposition 9.15 (4)]{RG} that
$$
d= \min \big\{ h \in \mathbb{N} \, \big| \, hg_l \in \langle g_1, \ldots, \widehat{g_l}, \ldots, g_\nu\rangle \big\}.
$$
By  unique expression of $\ap(S)$ we get $h g_l \notin \langle g_1, \ldots, \widehat{g_l}, \ldots, g_\nu\rangle$ for all $h \leq \alpha_l(S)$, so $\alpha_l(S) \leq d-1$.
On the other hand
$(\alpha_l(S)+1)g_l \notin \ap(S)$,
so it has another representation involving the multiplicity $g_1$,
and by maximality of $\alpha_l(S)$ this representation does not involve $g_l$.
Thus $(\alpha_l(S)+1) g_l \in \langle g_1, \ldots,  \widehat{g_l}, \ldots, g_\nu\rangle $ and $\alpha_l(S)+1 \geq d$.
Hence $d=\alpha_l(S)+1$.

Let us show now that $\alpha_{i}(T)\leq \alpha_{\sigma(i)}(S)$ for each $i=2,\ldots, \nu$
 (here $\alpha_i(T)$ denotes the integer $\alpha$ relative to the minimal generator $\frac{n_i}{d}$ of $T$).
If $\alpha_{i}(T)> \alpha_{\sigma(i)}(S)$,
then
$$(\alpha_{\sigma(i)}(S)+1) \frac{g_{\sigma(i)}}{d} \in \ap(T)
\, \Longrightarrow \,
(\alpha_{\sigma(i)}(S)+1) \frac{g_{\sigma(i)}}{d} - \frac{g_1}{d} \notin T
$$
because $m(T)=\frac{g_1}{d}$.
By definition of $\alpha_{\sigma(i)}$,
we have
$$
(\alpha_{\sigma(i)}(S)+1) g_{\sigma(i)} - g_1 \in S \, \Longrightarrow \, (\alpha_{\sigma(i)}(S)+1) g_{\sigma(i)} = g_1 + \sum_{j\ne l} \xi_j g_j + \xi_l g_l
$$
hence $d$ divides $\xi_l g_l$,
but $\gcd(d,g_l)=\gcd(S)=1$, therefore $d$ actually divides $\xi_l$.
It follows
\begin{equation*}\label{eq2}\tag{$ \dagger $}
(\alpha_{\sigma(i)}(S)+1) \frac{g_{\sigma(i)}}{d} = \frac{g_1}{d} + \sum_{j\ne l} \xi_j \frac{g_j}{d} + \frac{\xi_l}{d} g_l.
\end{equation*}
By definition of gluing $g_l=d_1\in T\setminus\big\{\frac{n_1}{d},\ldots, \frac{n_{\nu-1}}{d}\big\}$,
i.e. $g_l = \sum_{j\ne l} \eta_j \frac{g_j}{d}$.
Substituting this last equation in (\ref{eq2}) we obtain the contradiction
$$
(\alpha_{\sigma(i)}(S)+1) \frac{g_{\sigma(i)}}{d} - \frac{g_1}{d}
=
\sum_{j\ne l} \Big( \xi_j + \frac{\eta_j \xi_l}{d}\Big)\frac{g_j}{d}
\in T.
$$

Putting all the inequalities together, we get by Remark \ref{incl} and $\alpha$-rectangularity of $\ap(S)$
$$
m(T)\leq \prod_{i=2}^{\nu-1} \Big( \alpha_i(T)+1\Big) \leq \prod_{i\ne l } \Big( \alpha_{\sigma(i)}(S)+1\Big)= \frac{ \prod_{i=2  }^\nu \Big( \alpha_{i}(S)+1\Big)}{\alpha_l(S)+1}= \frac{m(S)}{d}=m(T)
$$
concluding that $\ap(T)$ is $\alpha$-rectangular by Proposition \ref{charalpha} $(v)$.

Now if $d_1=n_\nu \in \ap(T)$, then $n_\nu= \sum_{i=2}^{\nu-1}
\lambda_i \frac{n_i}{d}$ with $\lambda_i \leq \alpha_i(T)$ and
hence $n_\nu= \sum_{i=2}^{\nu-1} \lambda_{i}
\frac{g_{\sigma(i)}}{d}$ with $\lambda_{i} \leq
\alpha_{i}(T)\leq\alpha_{\sigma(i)}(S)$, by the previous part of
the proof. Since $\ap(S)$ is $\alpha$-rectangular, it follows that
$dn_\nu\in \ap(S)$, contradicting, { again  by the previous
part of the proof, the  definition of $\alpha_l$}.

The fact that $d_1=n_\nu>d_2 m(T)$ follows from $d_2 m(T)=n_1=g_1<g_l=n_\nu$.
\end{proof}

\begin{ex}
Let $T= \langle 18 , 21 , 27 , 35 \rangle$, then it is possible to check that $\alpha_2=2,\alpha_3=1, \alpha_4=2$ from which it follows that
$\ap(T)$ is $\alpha$-rectangular.
Let $S=2T+69\mathbb{N}= \langle 36, 42, 54, 69, 70\rangle$;
then we have $\alpha_2=2,\alpha_3=1, \alpha_4=3, \alpha_5=2$ so that 
$\ap(S)$ is not $\alpha$-rectangular.
This example shows the property of having $\alpha$-rectangular Ap\'ery set is not preserved by gluing with $\mathbb{N}$ if we drop the hypothesis $d_1 \notin \ap(T)$ in Theorem \ref{TheoremGlueAlpha} (notice that $69 \in \ap(T)$).
\end{ex}

\begin{que}
Is it possible
to characterize semigroups with $\beta$-rectangular and $\gamma$-rectangular Ap\'ery set in terms of gluing?
\end{que}

We conclude the paper by relating our work to a theorem of Watanabe and one of Rosales and Branco.
In \cite[Theorem 1]{Wa} the author proves that there exist complete intersection semigroups $S$ with prescribed values
of multiplicty and embedding dimension,
satisfying the condition $m(S) \geq 2^{\nu(S)-1}$.
We want to apply Theorem \ref{TheoremGlueAlpha} to prove a similar statement for semigroups with $\alpha$-rectangular Ap\'ery set.
However
we need the stronger condition $\ell(m(S))\geq \nu(S) - 1$,
 where $\ell(\cdot)$ denotes the length of the factorization into primes of an integer.
Note that this condition  is implied if $\ap(S)$ is  $\alpha$-rectangular,
as it follows from Propostion \ref{charalpha} $(v)$.

\begin{cor}
Given $m,\nu \in \mathbb{N}$ with $\ell(m)\geq \nu-1, \nu\geq 2$,
there exists a semigroup $S$ with $m(S)=m$, $\nu(S)=\nu$ such that $\ap(S)$ is $\alpha$-rectangular.
\end{cor}
\begin{proof}
Since $\ell(m)\geq \nu-1$ we can write  $m=a_1a_2\cdots a_{\nu-1}$,
with $a_i\geq 2$  integers, not necessarily prime.
Let
$S^{(1)}= \langle a_1, b \rangle$ where $b>a_1$ and $\gcd(a_1,b)=1$,
then
 $\ap(S^{(1)})$ is $\alpha$-rectangular.
Assume we constructed $S^{(j-1)}$ with $j>1$ fulfilling
 $m(S^{(j-1)})=a_1\cdots a_{j-1}$,
$\nu(S^{(j-1)})=j$
and with   $\ap(S^{(j-1)})$ $\alpha$-rectangular.
Glue
$S^{(j-1)}$ and $\mathbb{N}$ with integers $d_1$ and $d_2=a_j$,
choosing $d_1$ sufficiently large.
By Theorem \ref{TheoremGlueAlpha} the result $S^{(j)}$ has still $\alpha$-rectangular Ap\'ery set,
and furthermore $m(S^{(j)})=a_1\cdots a_j$ and $\nu (S^{(j)})=j+1$.
Finally take $S=S^{(\nu-1)}$.
\end{proof}

Now we analyze a family of semigroups introduced in \cite{RB},
where the authors provide families of free semigroups with arbitrary embedding dimension.

\begin{prop}
Let $a,b,p\in \mathbb{N}$ be such that $\gcd(a,b)=1$ and $a,b,p>1$.
The semigroup $S=\langle a^p, a^p+b, a^p+ab,\ldots, a^p+a^{p-1}b\rangle$ has $\alpha$-rectangular Ap\'ery set and is not telescopic.
\end{prop}
\begin{proof}
We have that $\{ a^p, a^p+b, a^p+ab,\ldots, a^p+a^{p-1}b \}$ 
are the minimal generators of $S$. 
Let us assume that $g_1=a^p, g_i = a^p + a^{i-2} b$ for all
$ i \in \{2, \ldots, p+1\}$.
If $i \leq p$,
we have $ag_i-g_1=a(a^p+a^{i-2}b)-a^p=(a-1)a^p+a^{i-1}b=g_{i+1}+(a-2)g_1\in S$.
If $i=p+1$
then
$ag_i-g_1=a(a^p+a^{p-1}b)-a^p=(a+b-1)g_1\in S$.
In both cases we have $ag_i\notin \ap(S)$ and hence $\alpha_i \leq a-1$.
Thus $g_1=a^p\geq \prod_{i=2}^{p+1}(\alpha_i+1)$.
But we have in general
 $g_1=a^p\leq \prod_{i=2}^{p+1}(\alpha_i+1)$ (cf. Remark \ref{incl})  and therefore $\ap(S)$ is $\alpha$-rectangular by Proposition \ref{charalpha} $(v)$.
  These semigroups are never telescopic:
 since $\gcd(a,b)=1$
we necessarily have $\tau_2+1=a^p$ and since $p>1$ it follows that $\prod_{i=2}^{p+1}(\tau_i+1)>\tau_2+1=a^p=m$.
\end{proof}


\begin{thebibliography}{99}


\bibitem{A}
 R. Ap\'ery, \emph{Sur les branches superlin\'eaires des courbes alg\'ebriques}, C. R. Acad. Sci. Paris {\bf 222}
(1946), 1198--1200.


\bibitem{AG} A. Assi, P. A. Garc\'ia Sanchez, {\em Constructing the set of complete intersection numerical semigroups with a given Frobenius number}, Applicable Algebra in Engineering, Communication and Computing
 {\bf 24} No. 2 (2013), 133--148.
 
 \bibitem{AM}
 F. Arslan, P. Mete, M. Sahin, \emph{Gluing and Hilbert functions of monomial curves}, Proceedings of the American Mathematical Society {\bf 137} No. 7 (2009),  2225--2232.


\bibitem{BDF} V. Barucci, M. D'Anna, R. Fr\"oberg, {\em On plane algebroid curves},
Commutative Ring Theory and Applications, Dekker Lecture Notes in Pure and
Applied Mathematics,  {\bf 231} (2003), 37--50.

\bibitem{BF} V. Barucci, R. Fr\"oberg, {\em
Associated graded rings of one dimensional analytically
irreducible rings}, Journal of Algebra {\bf 304} No. 1 (2006), 349--358.

\bibitem{BC} J. Bertin, P. Carbonne, {\em Semi-groupes d'entiers et application aux branches}, J. Algebra {\bf 49}
(1977), 81--95.

\bibitem{BGRV} I. Bermejo, P. Giminez, E. Reyes, R. H. Villarreal,
\emph{Complete intersections in affine monomial curves}, Boletin Sociedad Matem\'atica   Mexicana {\bf 3} No. 11 (2005), 191--203.

\bibitem{BGS} I. Bermejo, I. Garcia-Marco, J. J. Salazar-Gonzales,
\emph{An algorithm for checking whether the toric ideal of an affine monomial curve is a complete intersection}, J. Symbolic Computation
 {\bf 42} (2007), 971--991.

\bibitem{BCKR}
 C. Bowles, S. T. Chapman, N.  Kaplan, D. Reiser, \emph{On Delta Sets of Numerical Monoids}, Journal of Algebra and its Applications {\bf 5} (2006), 695--718.

\bibitem{Bre} H. Bresinsky, \emph{Semigroups corresponding to algebroid branches in the plane}, Proceedings of the American Mathematical Society {\bf 32} No. 2 (1972), 381-384.

\bibitem{Br}  L. Bryant, \emph{Goto numbers of a numerical semigroup ring and the
Gorensteiness of associated graded rings}, Communications in Algebra {\bf 38} No. 6 (2010), 2092--2128.


\bibitem{BH} L. Bryant, J. Hamblin, \emph{The maximal denumerant of a numerical semigroup}, 
Semigroup Forum {\bf 86} No. 3 (2013), 571--582. 

\bibitem{BHJ} L. Bryant, J. Hamblin, L. Jones, \emph{Maximal denumerant of a numerical semigroup with embedding
dimension less than four}, Journal of Commutative Algebra {\bf 4} No. 4 (2012), 489--503.


\bibitem{DMS2}
M. D'Anna, V. Micale, A. Sammartano, \emph{When the associated graded ring of a
semigroup ring is complete intersection},
Journal of Pure and Applied Algebra {\bf 217} (2013), 1007--1017.


\bibitem{NS} M. Delgado, P. A. Garcia-Sanchez, J. Morais, \emph{Numericalsgps, a   package for numerical semigroups}, Version 0.98 (2013).

\bibitem{De}  C. Delorme, \emph{Sous-mono\"ides d'intersection compl\`ete de N}, Ann. Sci.  Ecole Norm. Sup.
 {\bf 4} (1976), 145--154.

\bibitem{FGH}
R. Fr\"oberg, C. Gottlieb, R. H\"aggkvist, \emph{On numerical semigroups}, Semigroup forum  {\bf 35} No. 1 (1986), 63--83.

\bibitem{GAP} The GAP Group, \emph{GAP - Groups, Algorithms, and Programming},
Version 4.4.10 (2007).


\bibitem{GL}
P. A. Garcia-Sanchez, M. J. Leamer,
\emph{Huneke-Wiegand Conjecture for Complete Intersection Numerical Semigroup Rings}, Journal of Algebra {\bf 391} (2013),  114--124.




\bibitem{GOR} P. A. Garcia-Sanchez, I. Ojeda, J. C. Rosales,
\emph{Affine semigroups having a unique Betti element},
Journal of  Algebra and its  Applications {\bf 12} No. 3 (2013).

\bibitem{He} J. Herzog,
{\em Generators and relations of abelian semigroup rings}, Manuscripta Math. {\bf 3} (1970), 748--751.

\bibitem{JZ}
R. Jafari, S. Zarzuela Armengou,
\emph{On monomial curves obtained by gluing},
Semigroup Forum,
{\bf DOI} 10.1007/s00233-013-9536-1
 (2013).


\bibitem{KP} C. Kirfel, R. Pellikaan, {\em The minimum distance of codes in an array coming from telescopic semigroups}, IEEE Transactions on information theory, {\bf 41} No. 6,  (1995), 1720--1732.




\bibitem{MO} V. Micale, A. G. Olteanu, {\em On the Betti numbers of some semigroup rings}, Le Matematiche {\bf 67} (2012), 145--159.


\bibitem{Ro} J. C. Rosales, {\em Numerical semigroups with Ap\'ery sets of unique expression}, Journal of Algebra {\bf 226} No. 1
(2000), 479--487.

\bibitem{RB} J. C. Rosales, M. B. Branco, \emph{Three families of free numerical semigoups with arbitrary embedding dimension}, International Journal of Commutative Rings {\bf 1}  No. 4 (2002), 195--201.


\bibitem{RG} J. C. Rosales, P. A. Garc\'ia Sanchez, \emph{Numerical Semigroups}, Springer (2009).

\bibitem{RG2} J. C. Rosales, P. A. Garc\'ia Sanchez, \emph{On complete intersection affine semigroups}, Communications in Algebra {\bf 23} (1995), 5395--5412.

\bibitem{RGG} J. C. Rosales, P. A. Garc\'ia Sanchez, J. I. Garc\'ia-Garc\'ia, {\em k-factorized elements in telescopic
numerical semigroups},
Arithmetical Properties of Commutative Rings and Monoids {\bf 241} (2005), 260--271.

\bibitem{Wa} K. Watanabe, {\em Some examples of one dimensional Gorenstein domains}, Nagoya Mathematical Journal {\bf 49}
(1973), 101--109.

\bibitem{Za}
O. Zariski, \emph{Le probl\'eme des modules pour les branches planes}, Hermann, Paris, 1986.


\end{thebibliography}
\end{document}